\documentclass[11pt]{article}

\usepackage{amssymb,amsthm,amsmath}
\usepackage{mathrsfs}
\usepackage[dvipdfmx]{graphicx}
\usepackage{geometry}
\usepackage{color}

\makeatletter
    
    \@addtoreset{equation}{section}
\makeatother

\theoremstyle{definition}
\newtheorem{theorem}{Theorem}[section]

\newtheorem{definition}[theorem]{Definition}

\newtheorem{lemma}[theorem]{Lemma}
\newtheorem{remark}[theorem]{Remark}

\newcommand{\intint}{\int \hspace{-1.5mm} \int}
\newcommand{\intintint}{\int \hspace{-1.5mm} \int\hspace{-1.5mm} \int}
\newcommand{\inttx}{\intint_{\Omega_T}}
\newcommand{\eps}{\varepsilon}

\usepackage[T1]{fontenc}
\usepackage[utf8]{inputenc}
\usepackage{authblk}

\title{Energy conservation in the limit of filtered solutions for the 2D Euler equations}

\author{Takeshi Gotoda \footnote{Department of Mathematical and Computing Science, Tokyo Institute of Technology, 2-12-1 Ookayama, Meguro-ku, Tokyo, Japan, E-mail: gotoda@c.titech.ac.jp} }

\date{}

\begin{document}


\maketitle

\begin{abstract}
We consider energy conservation in a two-dimensional incompressible and inviscid flow through weak solutions of the filtered-Euler equations, which describe a regularized Euler flow based on a spatial filtering. We show that the energy dissipation rate for the filtered weak solution with vorticity in $L^p$, $p > 3/2$ converges to zero in the limit of the filter parameter. Although the energy defined in the whole space is not finite in general, we formally extract a time-dependent part, which is well-defined for filtered solutions, from the energy and define the energy dissipation rate as its time-derivative. Moreover, the limit of the filtered weak solution is a weak solution of the Euler equations and it satisfies a local energy balance in the sense of distributions. For the case of $p = 3/2$, we find the same result as $p > 3/2$ by assuming Onsager's critical condition for the family of the filtered solutions.

\end{abstract}


\section{Introduction} \label{sec:intro}

According to the Kolmogorov theory \cite{Kolmogorov}, energy dissipation in inviscid flows is closely related to three-dimensional (3D) turbulence. This implies that energy dissipating solutions of the 3D Euler equations are a key to comprehension of turbulent dynamics. Onsager conjectured that weak solutions of the 3D Euler equations acquiring a H\"{o}lder continuity with the order greater than $1/3$ conserve the energy, and the energy dissipation could occur for the order less than $1/3$ \cite{Eyink, Onsager, Shvydkoy}; Onsager's conjecture has been shown mathematically \cite{Buckmaster, Cheskidov(a), Constantin}. For 2D flows, the Kraichnan-Leith-Batchelor theory \cite{Batchelor, Kraichnan, Leith} indicates that two inertial ranges corresponding to a backward energy cascade and a forward enstrophy cascade appear in turbulent flows, which asserts that energy conservation in inviscid flows is still important for the 2D turbulent problem. In this paper, we study energy conserving solutions of an inviscid model.  Motions of incompressible and inviscid flows are often described by the 2D Euler equations:
\begin{equation}
\partial_t \boldsymbol{u} + (\boldsymbol{u} \cdot \nabla) \boldsymbol{u} + \nabla p = 0, \qquad \nabla \cdot \boldsymbol{u} = 0,  \label{eq:Euler}
\end{equation}
where $\boldsymbol{u}=\boldsymbol{u}(\boldsymbol{x}, t) =  ( u_1(\boldsymbol{x}, t), u_2(\boldsymbol{x}, t))$ is the fluid velocity field and $p=p(\boldsymbol{x}, t)$ is the scalar pressure. A classical weak solution for the initial value problem of \eqref{eq:Euler} with $\boldsymbol{u}(\boldsymbol{x}, 0) = \boldsymbol{u}_0(\boldsymbol{x})$ is defined as follows, see \cite{Diperna(b)}.

\begin{definition} \label{def:Euler}
{\it A velocity field $\boldsymbol{u} \in L^\infty(0, T; L^2_{\mathrm{loc}}(\mathbb{R}^2))$ vanishing at infinity is a weak solution of \eqref{eq:Euler} with initial data $\boldsymbol{u}_0$ provided that
\begin{itemize}
\item[$\mathrm{(i)}$] for any vector $\Psi \in C_c^\infty(\Omega_T)$ with $\nabla \cdot \Psi = 0$,
\begin{equation*}
\inttx \left( \partial_t \Psi \cdot \boldsymbol{u} + \nabla \Psi : \boldsymbol{u} \otimes \boldsymbol{u} \right) d\boldsymbol{x} dt = 0,
\end{equation*}
where $\Omega_T \equiv \mathbb{R}^2 \times (0,T)$, $\boldsymbol{v} \otimes \boldsymbol{v} = (v_i v_j)$, $\nabla \Psi = (\partial_j \psi_i)$ and $A : B = \sum_{i,j} a_{ij} b_{ij}$,
\item[$\mathrm{(ii)}$] for any scalar $\psi \in C_c^\infty(\Omega_T)$, $\inttx \nabla \psi \cdot \boldsymbol{u} d\boldsymbol{x} dt = 0$,
\item[$\mathrm{(iii)}$] $\boldsymbol{u} \in \mathrm{Lip}([0,T];H_{\mathrm{loc}}^{-L}(\mathbb{R}^2))$ for some $L > 0$ and $\boldsymbol{u}(\cdot, 0) = \boldsymbol{u}_0(\cdot)$ in $H_{\mathrm{loc}}^{-L}(\mathbb{R}^2)$,
\end{itemize}
}
\label{def-sol-velocity}
\end{definition}

Set vorticity $\omega \equiv \operatorname{curl} \boldsymbol{u} =  \partial_{x_1} u_2 - \partial_{x_2} u_1$. Taking the curl of \eqref{eq:Euler}, we obtain a transport equation for $\omega$:
\begin{equation}
\partial_t \omega + (\boldsymbol{u} \cdot \nabla) \omega = 0  \label{eq:vEuler}
\end{equation}
with initial vorticity $\omega_0 \equiv \operatorname{curl} \boldsymbol{u}_0$. The velocity $\boldsymbol{u}$ is recovered from $\omega$ via the Biot-Savart law:
\begin{equation}
\boldsymbol{u}(\boldsymbol{x},t) = \left( \boldsymbol{K} \ast \omega \right)(\boldsymbol{x},t) \equiv \int_{\mathbb{R}^2} \boldsymbol{K}(\boldsymbol{x} - \boldsymbol{y}) \omega(\boldsymbol{y},t) d \boldsymbol{y}, \label{eq:BSlaw}
\end{equation}
where $\boldsymbol{K}$ is defined by
\begin{equation*}
\boldsymbol{K}(\boldsymbol{x}) \equiv \nabla^\perp G(\boldsymbol{x}) = \frac{1}{2 \pi} \frac{\boldsymbol{x}^\perp}{\left| \boldsymbol{x} \right|^2}, \qquad G(\boldsymbol{x}) \equiv \frac{1}{2 \pi} \log{|\boldsymbol{x}|}
\end{equation*}
with $\nabla^\perp = (- \partial_{x_2}, \partial_{x_1})$ and $\boldsymbol{x}^\perp = (- x_2, x_1)$. Note that $G$ is a fundamental solution to the 2D Laplacian. In this paper, we focus on weak solutions of (\ref{eq:vEuler}) with $\omega_0 \in L^1(\mathbb{R}^2) \cap L^p(\mathbb{R}^2)$, $p > 1$; the existence of a global weak solution has been established for $1 < p \leq \infty$ and the uniqueness holds only for $p = \infty$ \cite{Diperna(b), Marchioro, Yudovich}. As it is mentioned in \cite{Lopes}, a weak solution for $\omega_0 \in L^1(\mathbb{R}^2) \cap L^p(\mathbb{R}^2)$, $p \geq 4/3$ satisfies \eqref{eq:vEuler} in the following sense.
\begin{equation*}
\inttx \left( \partial_t \psi(\boldsymbol{x},t) + \nabla \psi (\boldsymbol{x},t) \cdot \boldsymbol{u}(\boldsymbol{x},t) \right) \omega(\boldsymbol{x},t) d\boldsymbol{x} dt = 0
\end{equation*}
for any $\psi \in C_c^\infty(\Omega_T)$. For a weak solution of \eqref{eq:vEuler}, we consider the kinetic energy,
\begin{equation}
\frac{1}{2}\int |\boldsymbol{u}(\boldsymbol{x},t)|^2 d\boldsymbol{x}, \label{energy}
\end{equation}
where $\boldsymbol{u}$ is given by \eqref{eq:BSlaw}, though \eqref{energy} is not finite on the entire space $\mathbb{R}^2$ except for specific vorticity, see \cite{Diperna(b)} for an example. Cheskidov et al. \cite{Cheskidov(b)} have shown that a weak solution of the 2D Euler equations on the torus $\mathbb{T}^2$, for which \eqref{energy} is finite, conserves the energy for $\omega_0 \in L^{3/2}(\mathbb{T}^2)$ by using a spatial mollification. They have also shown energy conservation for the weak solution obtained by an inviscid limit of the 2D Navier-Stokes equations for $\omega_0 \in L^p(\mathbb{T}^2)$, $p > 1$, which is called a {\it physically realizable weak solution}. In this paper, we consider another regularization of the Euler equations, which we call the {\it filtered-Euler equations}, and show energy conservation on $\mathbb{R}^2$ in the limit of the regularization parameter. Although the energy is still infinite for the filtered inviscid model, we extract a finite time-dependent term formally and see the convergence of its time-derivative. We also show that the weak solution of the 2D filtered-Euler equations converges weakly to a weak solution of the 2D Euler equations and satisfies a local energy balance equation.


The filtered-Euler equations are given by
\begin{equation}
\partial_t \boldsymbol{v}^\varepsilon + (\boldsymbol{u}^\varepsilon \cdot \nabla) \boldsymbol{v}^\varepsilon - (\nabla \boldsymbol{v}^\varepsilon)^T \cdot \boldsymbol{u}^\varepsilon + \nabla p^\varepsilon = 0, \qquad \nabla \cdot \boldsymbol{v}^\varepsilon = 0,  \label{FEE}
\end{equation}
where $\boldsymbol{v}^\varepsilon$ and $p^\varepsilon$ denote the velocity field and the generalized pressure, respectively. Another field $\boldsymbol{u}^\varepsilon$ is a spatially filtered velocity of $\boldsymbol{v}^\varepsilon$, that is,
\begin{equation}
\boldsymbol{u}^\varepsilon(\boldsymbol{x},t) = \left( h^\varepsilon \ast \boldsymbol{v}^\varepsilon \right) (\boldsymbol{x},t), \qquad h^\varepsilon(\boldsymbol{x}) \equiv \frac{1}{\varepsilon^2} h \left( \frac{\boldsymbol{x}}{\varepsilon} \right), \qquad \eps > 0, \label{Fvelo}
\end{equation}
in 2D flows. Refer to \cite{Foias, Holm} for the derivation of the filtered-Euler equations through the filtering (\ref{Fvelo}). Here, $h \in L^1(\mathbb{R}^2)$ is a radial function satisfying $\int_{\mathbb{R}^2} h(\boldsymbol{x}) d\boldsymbol{x} = 1$, which we call the {\it filter function}. For simplicity, we assume $h \in C^1_0(\mathbb{R}^2 \setminus \{ 0 \})$: a continuously differentiable function that vanishes at infinity and may have a singularity at the origin. Note that, considering specific filter functions, we obtain two well-known regularizations: the Euler-$\alpha$ model and the vortex blob model, see \cite{G.1, Holm}. In particular, the Euler-$\alpha$ equations and their viscous extension, the Navier-Stokes-$\alpha$ equations, are considered as physically relevant models of turbulent flows \cite{Chen(a), Chen(b), Foias, Foias(a), Lunasin, Mohseni}.

Taking the $\operatorname{curl}$ of (\ref{FEE}) with the incompressible condition, we obtain the transport equation for $q^\varepsilon \equiv \operatorname{curl} \boldsymbol{v}^\varepsilon$ convected by $\boldsymbol{u}^\varepsilon$,
\begin{equation}
\partial_t q^\varepsilon + (\boldsymbol{u}^\varepsilon \cdot \nabla) q^\varepsilon = 0, \qquad \boldsymbol{u}^\varepsilon = \boldsymbol{K}^\varepsilon \ast q^\varepsilon, \qquad \boldsymbol{K}^\varepsilon \equiv \boldsymbol{K} \ast h^\varepsilon. \label{VFEE}
\end{equation}
The Biot-Savart law for the filtered vorticity $\omega^\varepsilon \equiv \operatorname{curl} \boldsymbol{u}^\varepsilon$ gives $\boldsymbol{u}^\varepsilon = \boldsymbol{K} \ast \omega^\varepsilon$ and we have $\nabla \cdot \boldsymbol{u}^\varepsilon = 0$ and $\omega^\varepsilon = h^\varepsilon \ast q^\varepsilon$ when the convolution commutes with the differential operator. The Lagrangian flow map $\boldsymbol{\eta}^\varepsilon$ associated with $\boldsymbol{u}^\varepsilon$ is given by
\begin{equation}
\partial_t \boldsymbol{\eta}^\varepsilon(\boldsymbol{x}, t) = \boldsymbol{u}^\varepsilon \left( \boldsymbol{\eta}^\varepsilon(\boldsymbol{x}, t), t \right), \qquad  \boldsymbol{\eta}^\varepsilon(\boldsymbol{x}, 0) = \boldsymbol{x}. \label{flow-map}
\end{equation}
The preceding study \cite{G.1} has shown that the 2D filtered-Euler equations have a unique global weak solution for $q_0 \in \mathcal{M}(\mathbb{R}^2)$, the space of finite Radon measures on $\mathbb{R}^2$, under some additional conditions for $h$. More precisely, we have a unique solution,
\begin{equation}
\boldsymbol{\eta}^\eps \in C^1([0,T]; \mathscr{G}), \qquad  q^\eps \in C_w([0,T]; \mathcal{M}(\mathbb{R}^2)), \qquad \boldsymbol{u}^\eps \in C([0,T]; C_0(\mathbb{R}^2)),  \label{sol:VFEE}
\end{equation}
to \eqref{VFEE} and \eqref{flow-map} with $q_0 \in \mathcal{M}(\mathbb{R}^2)$, where $\mathscr{G}$ denotes the group of all homeomorphisms of $\mathbb{R}^2$ preserving the Lebesgue measure and $C_w$ does the weak continuity. Note that $\boldsymbol{\eta}^\eps$, $q^\eps$ and $\boldsymbol{u}^\eps$ are related to each other: $q^\varepsilon(\boldsymbol{x}, t) = q_0 \left( \boldsymbol{\eta}^\varepsilon(\boldsymbol{x}, - t) \right)$ and $\boldsymbol{u}^\eps = \boldsymbol{K}^\eps \ast q^\eps$. The weak solution \eqref{sol:VFEE} satisfies (\ref{VFEE}) in the sense that
\begin{equation*}
\inttx \left( \partial_t \psi(\boldsymbol{x},t)  + \nabla \psi(\boldsymbol{x},t) \cdot \boldsymbol{u}^\varepsilon(\boldsymbol{x},t) \right) q^\varepsilon(\boldsymbol{x},t) d\boldsymbol{x} dt = 0
\end{equation*}
for any $\psi\in C_0^\infty(\Omega_T)$, the space of smooth functions vanishing at infinity in $\mathbb{R}^2$ and the boundary of $(0,T)$. We mention the convergence of weak solutions of the 2D filtered-Euler equations to those of the 2D Euler equations in the $\eps \rightarrow 0$ limit. For $q_0 \in L^1(\mathbb{R}^2) \cap L^\infty(\mathbb{R}^2)$, the weak solution of \eqref{VFEE} strongly converges to a unique global weak solution of \eqref{eq:Euler} with $\omega_0 = q_0$: the filtered flow map $\boldsymbol{\eta}^\eps$ converges to a flow map induced by the 2D Euler equations \cite{G.1}. For $q_0 \in L^1(\mathbb{R}^2) \cap L^p(\mathbb{R}^2)$, $1 < p < \infty$, as we show later, the filtered weak solution converges weakly to a weak solution of \eqref{eq:Euler}, which is constructed in \cite{Diperna(b)}. The convergence result has been extended to initial vorticity in $\mathcal{M}(\mathbb{R}^2) \cap H^{-1}_{\mathrm{loc}}(\mathbb{R}^2)$ with a distinguished sign \cite{G.2}.

Throughout this paper, we use the following notations. A open ball is denoted by $B_r \equiv \{ \boldsymbol{x} \in \mathbb{R}^2 \mid |\boldsymbol{x}| < r  \}$. For a set $A \subset \mathbb{R}^2$, $\chi_{A}$ denotes the indicator function and $|A|$ does the Lebesgue measure.  For the exponent $p$ in the Lebesgue or Sobolev space, $p'$ is the conjugate exponent of $p$, that is, $1 = 1/p + 1/{p'}$ for $p \in [1, \infty]$, and $p^\ast \in (2,\infty)$ is defined by $ p^\ast \equiv 2p/(2 - p)$, that is, $1/{p^\ast} = 1/p -1/2$ for $p \in (1, 2)$. We also introduce the weight function $w_{\alpha}(\boldsymbol{x}) \equiv  |\boldsymbol{x}|^\alpha$ for $\boldsymbol{x} \in \mathbb{R}^2$. Note that we omit the domain in the norm when it is the entire space $\mathbb{R}^2$. As for convergence, $f_n \rightarrow f$ denotes strong convergence and $f_n \rightharpoonup f$ does weak convergence in Banach spaces.

\section{Main results}

\subsection{Energy dissipation rate} \label{Subsec:FEE-dissipation}
Before deriving the energy dissipation rate for the filtered-Euler equations, we see basic properties of a weak solution to \eqref{VFEE} with $q_0 \in \mathcal{M}(\mathbb{R}^2)$. Considering the Lagrangian flow map $\boldsymbol{\eta}^\eps$, we find $\| q^\varepsilon(\cdot, t) \|_{\mathcal{M}} = \| q_0 \|_{\mathcal{M}}$ and $\{ q^\eps \}$ is uniformly bounded in $C([0,T]; \mathcal{M}(\mathbb{R}^2))$. We also have
\begin{equation*}
\| \omega^\eps(\cdot, t) \|_{L^p} \leq \| h^\eps \|_{L^p} \| q_0 \|_{\mathcal{M}} = \eps^{-2(1 - 1/p)} \| h \|_{L^p} \| q_0 \|_{\mathcal{M}}
\end{equation*}
for any $1 \leq p < \infty$, which implies that $\{ \omega^\varepsilon \} \subset C([0,T]; L^1(\mathbb{R}^2))$ is uniformly bounded. As for the filtered velocity $\boldsymbol{u}^\eps$, it follows from
\begin{equation*}
\boldsymbol{u}^\eps = \boldsymbol{K} \ast \omega^\eps = (\boldsymbol{K} \chi_{B_1}) \ast \omega^\eps + (\boldsymbol{K}\chi_{\mathbb{R}^2 \setminus B_1}) \ast \omega^\eps,
\end{equation*}
that
\begin{equation}
\| \boldsymbol{u}^\eps (\cdot, t) \|_{L^\infty} \leq C (\| \omega^\eps (\cdot, t) \|_{L^r} + \| \omega^\eps (\cdot, t) \|_{L^1} ) \leq C \eps^{-2(1 - 1/r)} \| q_0 \|_{\mathcal{M}}    \label{est:u-eps-infty}
\end{equation}
for any $2 < r \leq \infty$. As we see below, the above estimates give well-posedness of the energy dissipation rate.

We define energy for a filtered weak solution by replacing $\boldsymbol{u}$ with $\boldsymbol{u}^\eps$ in \eqref{energy}. However, this energy is not finite in general since the filtered Biot-Savart law, $\boldsymbol{u}^\eps = \boldsymbol{K}^\eps \ast q^\eps$, implies $\boldsymbol{u}^\eps(\boldsymbol{x}) \sim |\boldsymbol{x}|^{-1}$ as $|\boldsymbol{x}| \rightarrow \infty$. We now see that a formal calculation divides the energy into two parts: a time-invariant term and a time-dependent term. In particular, we focus on the time-dependent term that is well-defined for weak solutions of \eqref{VFEE} with $q_0 \in \mathcal{M}(\mathbb{R}^2)$. We start by substituting the filtered Biot-Savart law into the energy:
\begin{equation*}
\frac{1}{2}\int_{\mathbb{R}^2} \left| \boldsymbol{u}^\varepsilon (\boldsymbol{x},t)  \right|^2 d \boldsymbol{x} = \frac{1}{2} \intintint \boldsymbol{K}^\varepsilon(\boldsymbol{x} - \boldsymbol{y}) \cdot \boldsymbol{K}^\varepsilon(\boldsymbol{x} - \boldsymbol{z}) q^\varepsilon(\boldsymbol{y},t) q^\varepsilon(\boldsymbol{z},t)  d \boldsymbol{y} d \boldsymbol{z} d \boldsymbol{x}.
\end{equation*}
Since we have $\boldsymbol{K}^\varepsilon = \nabla^\perp G^\varepsilon$ and $\Delta G^\varepsilon = h^\varepsilon$, a formal calculation yields
\begin{align*}
\frac{1}{2} \int_{\mathbb{R}^2} \left| \boldsymbol{u}^\varepsilon (\boldsymbol{x},t)  \right|^2 d \boldsymbol{x} &= - \frac{1}{2} \intintint h^\varepsilon(\boldsymbol{x} - \boldsymbol{y}) \cdot G^\varepsilon(\boldsymbol{x} - \boldsymbol{z}) q^\varepsilon(\boldsymbol{y},t) q^\varepsilon(\boldsymbol{z},t) d \boldsymbol{x} d \boldsymbol{y} d \boldsymbol{z} \\
&= - \frac{1}{2} \intint \left( h^\varepsilon \ast G^\varepsilon \right)(\boldsymbol{y} - \boldsymbol{z}) q^\varepsilon(\boldsymbol{y},t) q^\varepsilon(\boldsymbol{z},t) d \boldsymbol{y} d \boldsymbol{z}.
\end{align*}
We introduce the following quantity.
\begin{equation*}
\mathscr{H}^\varepsilon \equiv - \frac{1}{2} \intint G^\varepsilon(\boldsymbol{x} - \boldsymbol{y}) q^\varepsilon(\boldsymbol{x},t) q^\varepsilon(\boldsymbol{y},t) d \boldsymbol{x} d \boldsymbol{y},
\end{equation*}
which is called the {\it pseudo-energy}. Although $\mathscr{H}^\varepsilon$ is not finite in general, considering specific vorticity, for example, initial vorticity of compact support, we find that $\mathscr{H}^\varepsilon$ is a conserved quantity. Indeed, for the point-vortex initial vorticity, $\mathscr{H}^\varepsilon$ gives the Hamiltonian of the {\it filtered point-vortex system}, see \cite{G.3}. On the basis of the above calculation, we divide the energy into two parts as follows.
\begin{equation*}
\frac{1}{2} \int_{\mathbb{R}^2} \left| \boldsymbol{u}^\varepsilon (\boldsymbol{x},t)  \right|^2 d \boldsymbol{x} = \mathscr{H}^\varepsilon + \mathscr{E}(t),
\end{equation*}
\begin{equation*}
\mathscr{E}(t) \equiv - \frac{1}{2} \intint H_G^\varepsilon (\boldsymbol{x} - \boldsymbol{y}) q^\varepsilon(\boldsymbol{x},t) q^\varepsilon(\boldsymbol{y},t) d \boldsymbol{x} d \boldsymbol{y},
\end{equation*}
where
\begin{equation}
H_G^\eps(\boldsymbol{x}) \equiv \left( h^\varepsilon \ast G^\varepsilon \right)(\boldsymbol{x}) - G^\varepsilon (\boldsymbol{x}) = \left( h^\varepsilon \ast \left(G^\varepsilon - G \right) \right)(\boldsymbol{x}). \label{def-H_G-eps}
\end{equation}
Refer to Appendix~\ref{appendix_A} for detailed properties of $H_G^\eps$ and $\nabla H_G^\eps$. As we see in Appendix~\ref{appendix_A}, $H_G^\eps$ belongs to $C_0(\mathbb{R}^2)$ for any fixed $\eps$, so that we find
\begin{equation*}
|\mathscr{E}(t)| \leq \| H_G^\eps \|_{L^\infty} \| q^\eps(\cdot, t) \|_{\mathcal{M}}^2  = C_\eps \| q_0 \|_{\mathcal{M}}^2,
\end{equation*}
where $C_\eps$ is the constant depending on $\eps$. Thus, the time-dependent term $\mathscr{E}(t)$ is finite for any $q_0 \in \mathcal{M}(\mathbb{R}^2)$. Since we have
\begin{equation*}
\intint H_G^\varepsilon (\boldsymbol{x} - \boldsymbol{y}) q^\varepsilon(\boldsymbol{x},t) q^\varepsilon(\boldsymbol{y},t) d \boldsymbol{y} d \boldsymbol{x} = \intint H_G^\varepsilon (\boldsymbol{\eta}^\varepsilon(\boldsymbol{x}, t) - \boldsymbol{\eta}^\varepsilon(\boldsymbol{y}, t)) q_0(\boldsymbol{x}) q_0(\boldsymbol{y}) d \boldsymbol{x} d \boldsymbol{y},
\end{equation*}
the time-derivative of $\mathscr{E}(t)$ is given by
\begin{align*}
\frac{\mbox{d}}{\mbox{d}t} \mathscr{E}(t) &= - \frac{1}{2} \intint (\nabla H_G^\varepsilon) (\boldsymbol{\eta}^\varepsilon(\boldsymbol{x}, t) - \boldsymbol{\eta}^\varepsilon(\boldsymbol{y}, t)) \\
&\hspace{20mm} \cdot \left( \boldsymbol{u}^\varepsilon(\boldsymbol{\eta}^\varepsilon(\boldsymbol{x}, t),t) - \boldsymbol{u}^\varepsilon(\boldsymbol{\eta}^\varepsilon(\boldsymbol{y}, t),t) \right) q_0(\boldsymbol{x}) q_0(\boldsymbol{y}) d \boldsymbol{x} d \boldsymbol{y}\\
&= - \frac{1}{2} \intint (\nabla H_G^\varepsilon) (\boldsymbol{x} - \boldsymbol{y}) \cdot \left( \boldsymbol{u}^\varepsilon(\boldsymbol{x}, t) - \boldsymbol{u}^\varepsilon(\boldsymbol{y}, t) \right) q^\varepsilon(\boldsymbol{x},t) q^\varepsilon(\boldsymbol{y},t) d \boldsymbol{x} d \boldsymbol{y}.
\end{align*}
Hence, we define the energy dissipation rate by
\begin{equation*}
\mathscr{D}_E^\varepsilon(t) \equiv - \frac{1}{2} \intint (\nabla H_G^\varepsilon) (\boldsymbol{x} - \boldsymbol{y}) \cdot \left( \boldsymbol{u}^\varepsilon(\boldsymbol{x}, t) - \boldsymbol{u}^\varepsilon(\boldsymbol{y}, t) \right) q^\varepsilon(\boldsymbol{x},t) q^\varepsilon(\boldsymbol{y},t) d \boldsymbol{x} d \boldsymbol{y}.
\end{equation*}
It follows from $\nabla H_G^\varepsilon \in C_0(\mathbb{R}^2)$, see Appendix~\ref{appendix_A}, and \eqref{est:u-eps-infty} that
\begin{equation*}
|\mathscr{D}_E^\varepsilon(t)| \leq \| \nabla H_G^\varepsilon \|_{L^\infty} \| \boldsymbol{u}^\varepsilon(\cdot, t) \|_{L^\infty} \| q^\eps (\cdot, t) \|_{\mathcal{M}}^2 \leq C_\eps \| q_0 \|_{\mathcal{M}}^3,
\end{equation*}
and thus $\mathscr{D}_E^\varepsilon$ is well-defined for weak solutions of \eqref{VFEE} with $q_0 \in \mathcal{M}(\mathbb{R}^2)$.

\subsection{Main theorems} \label{Main}

As we see in Section~\ref{Subsec:FEE-dissipation}, the energy dissipation rate $\mathscr{D}_E^\varepsilon(t)$ is bounded for any weak solution of \eqref{VFEE} with $q_0 \in \mathcal{M}(\mathbb{R}^2)$. However, the boundedness of $\mathscr{D}_E^\varepsilon$ depends on the filter parameter $\eps$ and, in the $\eps \rightarrow 0$ limit, $\mathscr{D}_E^\varepsilon$ is not finite in general. Our concern is the set of initial vorticity that provides the uniform boundedness of $\{ \mathscr{D}_E^\varepsilon \}$. In this paper, we consider weak solutions of \eqref{VFEE} with $q_0 \in L^1(\mathbb{R}^2) \cap L^p(\mathbb{R}^2)$, and give a sufficient condition for $p$ that yields energy conservation: $\mathscr{D}_E^\varepsilon (t) \rightarrow 0$ as $\eps \rightarrow 0$. In the following theorems, we assume that the filter function $h$ is sufficiently regular, so that \eqref{VFEE} has a unique global weak solution for $q_0 \in \mathcal{M}(\mathbb{R}^2)$, see \cite{G.1} for a sufficient condition for $h$.

\begin{theorem} \label{thm-main1}
{\it
Suppose that $h \in C_0^1(\mathbb{R}^2 \setminus \{ 0 \})$ is a radial function satisfying
\begin{equation*}
w_1 h, \ \nabla h \in L^1(\mathbb{R}^2), \qquad w_\alpha h, \ w_3 h, \ w_1 \nabla h \in L^\infty(\mathbb{R}^2)
\end{equation*}
for some $\alpha \in [0,1)$. Let $(\boldsymbol{u}^\eps, q^\eps)$ be a weak solution of the 2D filtered-Euler equations with $q_0 \in L^1(\mathbb{R}^2) \cap L^p(\mathbb{R}^2)$, $ 3/2 < p \leq \infty$. Then, we have
\begin{equation*}
\lim_{\varepsilon \rightarrow 0} \| \mathscr{D}_E^\varepsilon \|_{L^\infty(0,T)} = 0.
\end{equation*}
Moreover, there exists a weak solution of the 2D Euler equations,
\begin{equation*}
\boldsymbol{u} \in L^\infty (0,T;L^{p^\ast}(\mathbb{R}^2) \cap W^{1,p}_{\mathrm{loc}}(\mathbb{R}^2)), \quad \omega = \operatorname{curl} \boldsymbol{u} \in L^\infty (0,T; L^1(\mathbb{R}^2) \cap L^p(\mathbb{R}^2)),
\end{equation*}
such that, taking subsequences as needed, we have
\begin{equation*}
q^\eps \rightharpoonup \omega \quad \mathrm{in}\ L^p (\Omega_T),\qquad  \boldsymbol{u}^\eps \rightarrow \boldsymbol{u} \quad \mathrm{in}\ C([0,T];L^r_{\mathrm{loc}} (\mathbb{R}^2))
\end{equation*}
for any $r \in [1, p^\ast)$ in the $\eps \rightarrow 0$ limit, and there exists $P \in L^\infty(0,T;L^{p^\ast/2}(\mathbb{R}^2))$ such that the following local energy balance holds in the sense of distributions.
\begin{equation*}
\partial_t \left( \frac{|\boldsymbol{u}|^2}{2} \right) + \nabla \cdot \left(\boldsymbol{u} \left( \frac{|\boldsymbol{u}|^2}{2} + P \right) \right) = 0.
\end{equation*}
}
\end{theorem}

 The conditions for $h$ in Theorem~\ref{thm-main1} imply $h \in L^p(\mathbb{R}^2)$ and $\nabla h \in L^q(\mathbb{R}^2)$ for any $p \in [1,\infty)$ and $q \in [1,2)$. The filter functions for the Euler-$\alpha$ model and the vortex blob model satisfy these conditions. The convergence to the Euler equations for $q_0 \in L^1(\mathbb{R}^2) \cap L^p(\mathbb{R}^2)$ is proven in the same way as \cite{G.2}, but this paper gives a simpler proof for it. As mentioned in the introduction, Cheskidov et al. \cite{Cheskidov(b)} have shown that a weak solution of the 2D Euler equations on $\mathbb{T}^2$ conserves the energy and satisfies the local energy balance provided that its vorticity belongs to $L^p(\mathbb{T}^2)$, $p \geq 3/2$. In considering the $\eps \rightarrow 0$ limit in Theorem~\ref{thm-main1}, the condition $p > 3/2$ is essential for its proof. For the case of $p = 3/2$, however, the same result as Theorem~\ref{thm-main1} holds under an additional condition for the regularity of $\boldsymbol{u}^\eps$:

\begin{theorem} \label{thm-main2}
{\it
Let $(\boldsymbol{u}^\eps, q^\eps)$ be a weak solution of the 2D filtered-Euler equations with $q_0 \in L^1(\mathbb{R}^2) \cap L^{3/2}(\mathbb{R}^2)$ and $\boldsymbol{u}^\eps$ satisfy
\begin{equation}
\| \boldsymbol{u}^\eps(\cdot - \boldsymbol{y},t) - \boldsymbol{u}^\eps(\cdot,t) \|_{L^3} \leq C(T) |\boldsymbol{y}|^\alpha, \quad (\boldsymbol{y},t) \in \Omega_T, \label{condi:thm2}
\end{equation}
for some $\alpha \in (1/3,1]$, where $C(T)$ is independent of $\eps$. Then, we have the same result as Theorem~\ref{thm-main1} with $p = 3/2$.
}
\end{theorem}

We remark that \eqref{condi:thm2} is related to Onsager's critical condition, that is, $1/3$-H\"{o}lder continuity. Although Onsager conjectured for the 3D Euler equations, the energy conservation holds for the weak solution of the Euler equations satisfying \eqref{condi:thm2} regardless of the dimension \cite{Cheskidov(a)}. As it is mentioned in \cite{Cheskidov(b)}, weak solutions of the 2D Euler equations with $\omega_0 \in L^{3/2}$ satisfy \eqref{condi:thm2}. Our main theorems are consistent with these preceding results, though we require the family $\{ \boldsymbol{u}^\eps \}$ to satisfy \eqref{condi:thm2} uniformly: the existence of a uniform constant $C(T)$ with respect to $\eps > 0$.

\section{Proof of main theorems}

It is sufficient to show Theorem~\ref{thm-main1} for $3/2 < p < 2$. In what follows, let $p \in (3/2,2)$ be a fixed constant.

\subsection{Convergence to the Euler equations} \label{proof:conv-Euler}

Consider a weak solution of \eqref{VFEE} with $q_0 \in L^1(\mathbb{R}^2) \cap L^p(\mathbb{R}^2)$. For any $q \in [1, p]$, we have $\| q^\eps(\cdot,t) \|_{L^q} = \| q_0 \|_{L^q}$ and
\begin{equation*}
\| \omega^\eps(\cdot,t) \|_{L^q} \leq \| h \|_{L^1} \| q_0 \|_{L^q},
\end{equation*}
so that $\{ q^\eps \}$ and $\{ \omega^\eps \}$ are uniformly bounded in $C([0,T];L^q(\mathbb{R}^2))$. Thus, there exists $\omega \in L^p(\Omega_T)$ such that $q^\eps$, $\omega^\eps \rightharpoonup \omega$ in $L^{p}(\Omega_T)$ by taking subsequences as needed since we easily find $(q^\eps -\omega^\eps)  \rightharpoonup 0$ in $L^{p}(\Omega_T)$. As for the filtered velocity $\boldsymbol{u}^\eps$, it follows from the Hardy-littlewood-Sobolev inequality that
\begin{equation}
\| \boldsymbol{u}^\eps (\cdot, t) \|_{L^{p^\ast}} \leq C \| \omega^\eps(\cdot,t) \|_{L^p} \leq C \| q_0 \|_{L^p}.   \label{ineq:HLS}
\end{equation}
More generally, we have
\begin{equation}
\| \boldsymbol{u}^\eps (\cdot, t) \|_{L^r} \leq C \| \omega^\eps(\cdot,t) \|_{L^s} \label{ineq:GHLS}
\end{equation}
for any $r \in (2, p^\ast]$ and $s \in (1,p]$ satisfying $1/r = 1/s - 1/2$. It follows from the Calder\'{o}n-Zygmund inequality that
\begin{equation}
\| \nabla \boldsymbol{u}^\eps (\cdot,t)\|_{L^q} \leq C \| \omega^\eps(\cdot,t) \|_{L^q} \label{ineq:CZ}
\end{equation}
for any $q \in (1,\infty)$, which yields $\| \nabla \boldsymbol{u}^\eps (\cdot,t)\|_{L^p} \leq C \| q_0 \|_{L^p}$. Since we have
\begin{align*}
\| \boldsymbol{u}^\eps (\cdot,t) \|_{L^p(B_R)} &\leq \| \boldsymbol{K} \chi_{B_1} \ast \omega^\eps (\cdot,t) \|_{L^p} + \| \boldsymbol{K} \chi_{\mathbb{R}^2 \setminus B_1} \ast \omega^\eps (\cdot,t) \|_{L^p(B_R)} \\
&\leq \| \boldsymbol{K} \chi_{B_1} \|_{L^1} \| \omega^\eps (\cdot,t) \|_{L^p} +  |B_R|^{1/p} \| \boldsymbol{K} \chi_{\mathbb{R}^2 \setminus B_1} \|_{L^{p'}} \| \omega^\eps (\cdot,t) \|_{L^p} \\
&\leq C _R \| \omega^\eps (\cdot,t) \|_{L^p} \leq C_R \| q_0\|_{L^p}
\end{align*}
for any $R > 0$, $\{ \boldsymbol{u}^\eps \}$ is uniformly bounded in $C([0,T]; W^{1,p}_{\mathrm{loc}}(\mathbb{R}^2))$. Note that $\omega^\eps$ satisfies
\begin{equation*}
\inttx \left( (\partial_t \psi) \omega^\varepsilon + \nabla \psi \cdot \left( h^\varepsilon \ast (\boldsymbol{u}^\varepsilon q^\varepsilon) \right) \right) d\boldsymbol{x} dt = 0
\end{equation*}
for any $\psi\in C_c^\infty(\Omega_T)$ and it follows from $\omega^\eps = \nabla^\perp \cdot \boldsymbol{u}^\eps$ that
\begin{equation}
\inttx \left( (\partial_t \nabla^\perp \psi) \cdot \boldsymbol{u}^\varepsilon - \nabla^\perp \psi(\boldsymbol{x},t) \cdot ( h^\varepsilon \ast ((\boldsymbol{u}^\varepsilon)^\perp q^\varepsilon) )  \right) d\boldsymbol{x} dt = 0. \label{eq:weak-velocity}
\end{equation}
Then, we obtain
\begin{align*}
\left| \inttx  \nabla^\perp \psi \cdot \partial_t \boldsymbol{u}^\varepsilon d\boldsymbol{x} dt \right| &\leq \int_0^T \| \nabla^\perp \psi(\cdot,t) \|_{L^\infty} \| \boldsymbol{u}^\varepsilon (\cdot,t) \|_{L^{p^\ast}} \| q^\varepsilon(\cdot,t) \|_{L^{(p^\ast)'}} dt \\
& \leq C \| \nabla^\perp \psi \|_{L^1(0,T; L^\infty(\mathbb{R}^2))} \| q_0 \|_{L^p} \| q_0 \|_{L^{(p^\ast)'}} \\
& \leq C(q_0) \| \nabla^\perp \psi \|_{L^1(0,T; H^2(\mathbb{R}^2))},
\end{align*}
where $C(q_0)$ is the constant depending on $\| q_0 \|_{L^1}$ and $\| q_0 \|_{L^p}$. Considering $\nabla \cdot \boldsymbol{u}^\eps = 0$, we find that $\{ \partial_t \boldsymbol{u}^\eps \}$ is uniformly bounded in $L^\infty(0,T; H^{-2}_{\mathrm{loc}}(\mathbb{R}^2))$. Note that the embedding $W^{1,p}(B_R)\hookrightarrow L^r(B_R)$ is compact for any $r \in [1, p^\ast)$. There exists $\boldsymbol{u} \in C([0,T]; L^r(B_R))$ such that, by taking subsequences as needed, $\boldsymbol{u}^\eps \rightarrow \boldsymbol{u}$ in $C([0,T]; L^r(B_R))$ for any $R >0$. In addition, the uniform estimates for $\{ \boldsymbol{u}^\eps \}$ yield $\boldsymbol{u} \in L^\infty (0,T; L^{p^\ast}(\mathbb{R}^2)\cap W^{1,p}_{\mathrm{loc}}(\mathbb{R}^2))$ and the Biot-Savart law $\boldsymbol{u} = \boldsymbol{K} \ast \omega$ holds. Recall that $(\boldsymbol{u}^\eps, q^\eps)$ satisfies
\begin{equation*}
\inttx \left( \partial_t \psi  + \boldsymbol{u}^\varepsilon \cdot \nabla \psi \right) q^\varepsilon d\boldsymbol{x} dt = 0
\end{equation*}
for any $\psi\in C_c^\infty(\Omega_T)$. The weak convergence of $q^\eps$ yields
\begin{equation*}
\inttx (\partial_t \psi) q^\varepsilon d\boldsymbol{x} dt  \ \longrightarrow\ \inttx (\partial_t \psi) \omega d\boldsymbol{x} dt.
\end{equation*}
The nonlinear term is divided into two parts as follows.
\begin{equation*}
\inttx \boldsymbol{u}^\varepsilon \cdot (\nabla \psi) q^\varepsilon d\boldsymbol{x} dt = \inttx \boldsymbol{u} \cdot (\nabla \psi) q^\varepsilon d\boldsymbol{x} dt + \inttx \left( \boldsymbol{u}^\varepsilon - \boldsymbol{u} \right) \cdot (\nabla \psi) q^\varepsilon d\boldsymbol{x} dt.
\end{equation*}
It follows from $\boldsymbol{u} \cdot \nabla \psi \in L^{p'}(\Omega_T)$ that
\begin{equation*}
\inttx \boldsymbol{u} \cdot (\nabla \psi) q^\varepsilon d\boldsymbol{x} dt \ \longrightarrow\ \inttx \boldsymbol{u} \cdot (\nabla \psi) \omega d\boldsymbol{x} dt.
\end{equation*}
Since there exist $r \in (3, p^\ast)$ and $s \in[1,\infty]$ such that $1 = 1/r + 1/p + 1/s$, we have
\begin{equation*}
\left| \inttx \left( \boldsymbol{u}^\varepsilon - \boldsymbol{u} \right) \cdot (\nabla \psi) q^\varepsilon d\boldsymbol{x} dt \right| \leq \| \boldsymbol{u}^\varepsilon - \boldsymbol{u} \|_{L^\infty(0,T; L^r(B_R))} \| \nabla \psi \|_{L^1(0,T; L^s(\mathbb{R}^2))} \| q_0 \|_{L^p}
\end{equation*}
for any $R > 0$ satisfying $\operatorname{supp} \psi \subset B_R \times (0,T)$, and the right-hand side converges to zero as $\eps \rightarrow 0$. Hence, we obtain
\begin{equation*}
\inttx \left( \partial_t \psi  + \boldsymbol{u} \cdot \nabla \psi \right) \omega d\boldsymbol{x} dt = 0,
\end{equation*}
that is, $(\boldsymbol{u}, \omega)$ is a weak solution of the 2D Euler equations.

\begin{remark}
The above proof for the convergence to the Euler equations is valid for $p$ satisfying $(p^\ast)' \leq p$, that is, $ p \geq 4/3$.
\end{remark}

\subsection{Convergence of the energy dissipation rate}
From the definition of $\mathscr{D}_E^\varepsilon$, we find
\begin{align*}
| \mathscr{D}_E^\varepsilon(t) | &\leq \frac{1}{2} \intint |(\nabla H_G^\varepsilon) (\boldsymbol{y})|  | \boldsymbol{u}^\varepsilon(\boldsymbol{x}, t) - \boldsymbol{u}^\varepsilon(\boldsymbol{x} - \boldsymbol{y}, t) | | q^\varepsilon(\boldsymbol{x},t) | |q^\varepsilon(\boldsymbol{x} - \boldsymbol{y},t)| d \boldsymbol{x} d \boldsymbol{y} \\
& \leq \frac{1}{2} \int |(\nabla H_G^\varepsilon) (\boldsymbol{y})| \| \boldsymbol{u}^\varepsilon(\cdot, t) - \boldsymbol{u}^\varepsilon(\cdot - \boldsymbol{y}, t) \|_{L^{p'}} \| q^\varepsilon(\cdot,t) q^\varepsilon(\cdot - \boldsymbol{y},t) \|_{L^p} d \boldsymbol{y}.
\end{align*}
Note that
\begin{equation*}
\| \boldsymbol{u}^\eps(\cdot,t) - \boldsymbol{u}^\eps(\cdot - \boldsymbol{y},t) \|_{L^{p'}} \leq C |\boldsymbol{y}| \|\nabla \boldsymbol{u}^\eps (\cdot,t)\|_{L^{p'}},
\end{equation*}
and it follows from \eqref{ineq:CZ} that
\begin{equation*}
\| \nabla \boldsymbol{u}^\eps (\cdot,t) \|_{L^{p'}}  \leq C \| \omega^\eps (\cdot,t) \|_{L^{p'}} \leq  C \eps^{-2 + 2/s }\| h \|_{L^s} \| q_0 \|_{L^p},
\end{equation*}
where $s \in (1, 3/2)$ satisfies $1 + 1/{p'} = 1/s + 1/p$, that is, $1/s = 2 - 2/p$. Thus, we find
\begin{align*}
| \mathscr{D}_E^\varepsilon(t) | & \leq  C \eps^{2 - 4/p }\| h \|_{L^s} \| q_0 \|_{L^p} \int |\boldsymbol{y}| |\nabla H_G^\varepsilon (\boldsymbol{y})|  \| q^\varepsilon(\cdot,t) q^\varepsilon(\cdot - \boldsymbol{y},t) \|_{L^p} d \boldsymbol{y} \\
&\leq C \eps^{2 - 4/p }\| h \|_{L^s} \| q_0 \|_{L^p}^3 \| w_1 \nabla H_G^\varepsilon \|_{L^{p'}}.
\end{align*}
According to \eqref{est:dH_gr-eps}, we have
\begin{equation*}
|(\nabla H_G^\varepsilon) (\boldsymbol{x})| \leq \frac{C}{|\boldsymbol{x}|} \frac{\eps}{|\boldsymbol{x}| + \eps}, \label{est:HG}
\end{equation*}
where $C$ is the constant independent of $\eps$, and it follows that
\begin{align*}
\| w_1 \nabla H_G^\varepsilon \|_{L^{p'}} &\leq C \left( \int_{\mathbb{R}^2}  \frac{1}{(|\boldsymbol{x}|/\eps + 1)^{p'}} d\boldsymbol{x} \right)^{1/{p'}}  = C \eps^{2/{p'}} \left(  \int_0^\infty \frac{r}{(r + 1)^{p'}} dr \right)^{1/{p'}}.
\end{align*}
Owing to $2< p' <3$, the right-hand side is finite. Hence, we obtain
\begin{equation*}
| \mathscr{D}_E^\varepsilon(t) | \leq C \eps^{6(2/3 - 1/p) }  \| q_0 \|_{L^p}^3,
\end{equation*}
and $\| \mathscr{D}_E^\varepsilon \|_{L^\infty(0,T)}$ converges to zero in the $\eps \rightarrow 0$ limit.

\begin{remark} \label{rem:conv-Euler}
It is noteworthy that
\begin{equation*}
| \mathscr{D}_E^\varepsilon(t)| \leq C \| \boldsymbol{u}^\varepsilon(\cdot, t)\|_{L^{p^\ast}} \| q^\varepsilon(\cdot,t)\|_{L^p} \left\| \frac{1}{|\cdot|} \ast | q^\varepsilon(\cdot,t) | \right\|_{L^{p^\ast}}
\end{equation*}
holds for $p$ satisfying $1 = 2 /p^{\ast} + 1/p$, that is, $p =3/2$. Considering the Hardy-Littlewood-Sobolev inequality. we find
\begin{equation*}
| \mathscr{D}_E^\varepsilon(t)| \leq C \| \boldsymbol{u}^\varepsilon(\cdot, t)\|_{L^6} \| q^\varepsilon(\cdot,t)\|_{L^{3/2}}^2 \leq C \| q_0 \|_{L^{3/2}}^3,
\end{equation*}
so that $\|\mathscr{D}_E^\varepsilon \|_{L^\infty(0,T)}$ is uniformly bounded with respect to $\eps$.
\end{remark}

\subsection{Local energy balance} \label{proof:balance}
We show that the limit function $\boldsymbol{u}$ satisfies the local energy balance. If follows from \eqref{eq:weak-velocity} and $\nabla \cdot \boldsymbol{u}^\eps = 0$ that
\begin{align*}
&\inttx \left( \partial_t \nabla^\perp \psi \cdot \boldsymbol{u}^\varepsilon +  \nabla \otimes \nabla^\perp \psi :  \boldsymbol{u}^\varepsilon \otimes \boldsymbol{u}^\varepsilon \right)  d\boldsymbol{x} dt \\
&= \inttx \nabla \psi \cdot \left( h^\eps \ast (\boldsymbol{u}^\varepsilon q^\varepsilon) - \boldsymbol{u}^\varepsilon \omega^\varepsilon  \right)  d\boldsymbol{x} dt \equiv I
\end{align*}
for any $\psi\in C_c^\infty(\Omega_T)$. Since $I$ is rewritten by
\begin{equation*}
I = \inttx \nabla^\perp \psi \cdot \left( h^\eps \ast ((\boldsymbol{u}^\varepsilon)^\perp q^\varepsilon) - (\boldsymbol{u}^\varepsilon)^\perp \omega^\varepsilon  \right)  d\boldsymbol{x} dt,
\end{equation*}
there exists a distribution $P^\eps$ such that $\boldsymbol{u}^\eps$ satisfies
\begin{equation}
\partial_t \boldsymbol{u}^\varepsilon + (\boldsymbol{u}^\varepsilon \cdot \nabla ) \boldsymbol{u}^\varepsilon = - \nabla P^\eps  - \left( h^\eps \ast ((\boldsymbol{u}^\varepsilon)^\perp q^\varepsilon) - (\boldsymbol{u}^\varepsilon)^\perp \omega^\varepsilon  \right), \label{eq:u-eps}
\end{equation}
in the sense of distributions.

We first see properties of $P^\eps$. For any $r \in (2, p^\ast)$, it follows from \eqref{ineq:GHLS} and \eqref{ineq:CZ} that
\begin{equation}
\| \boldsymbol{u}^\eps (\cdot, t) \|_{L^r} \leq C \| \omega^\eps (\cdot, t) \|_{L^{s_1}} \leq C \| q_0 \|_{L^{s_1}}  \label{est:local-1}
\end{equation}
for $s_1 \in (1, p)$ satisfying $1/r = 1/{s_1} - 1/2$, and
\begin{align}
\| \nabla \boldsymbol{u}^\eps(\cdot, t) \|_{L^r} &\leq C \| \omega^\eps(\cdot, t) \|_{L^r} \leq C \eps^{-2 + 2/{s_2}} \| h \|_{L^{s_2}} \| q_0 \|_{L^p}, \label{est:local-2} \\
\| \nabla^2 \boldsymbol{u}^\eps(\cdot, t) \|_{L^r} &\leq C \| \nabla \omega^\eps(\cdot, t) \|_{L^r} \leq C \eps^{-3 + 2/{s_2}} \| \nabla h \|_{L^{s_2}} \| q_0 \|_{L^p}  \label{est:local-3}
\end{align}
for $s_2 \in (1,2)$ satisfying $1 + 1/r = 1/{s_2} + 1/p$, respectively. Thus, for any fixed $\eps$, we have $\boldsymbol{u}^\eps \in C([0,T]; W^{2,r} (\mathbb{R}^2))$, so that the Sobolev embedding gives $\boldsymbol{u}^\eps \in C([0,T]; C^1 (\mathbb{R}^2))$ and $\omega^\eps \in C(\overline{\Omega_T})$. Taking the divergence of \eqref{eq:u-eps}, we have
\begin{equation*}
\Delta P^\eps =  \nabla^\perp \cdot \left[ h^\eps \ast (\boldsymbol{u}^\varepsilon q^\varepsilon) - \boldsymbol{u}^\varepsilon \omega^\varepsilon - (\boldsymbol{u}^\eps \cdot \nabla) (\boldsymbol{u}^\varepsilon)^\perp \right] \equiv \nabla^\perp \cdot  \boldsymbol{F}^\eps,
\end{equation*}
which gives $P^\eps(\boldsymbol{x},t) = \int \boldsymbol{K}(\boldsymbol{x} - \boldsymbol{y}) \cdot \boldsymbol{F}^\eps(\boldsymbol{y},t) d\boldsymbol{y}$. Similarly to \eqref{est:local-1}, \eqref{est:local-2} and \eqref{est:local-3}, we have
\begin{align*}
\| P^\eps (\cdot, t) \|_{L^r} &\leq C \| \boldsymbol{F}^\eps (\cdot, t) \|_{L^{s_1}}, \\
\| \nabla P^\eps (\cdot, t) \|_{L^r} &\leq C \| \boldsymbol{F}^\eps (\cdot, t) \|_{L^r}, \\
\| \nabla^2 P^\eps (\cdot, t) \|_{L^r} &\leq C \| \nabla \boldsymbol{F}^\eps (\cdot, t) \|_{L^r},
\end{align*}
and it follows that
\begin{align*}
\| \boldsymbol{F}^\eps (\cdot, t) \|_{L^1} &\leq \| \boldsymbol{u}^\eps(\cdot, t) \|_{L^{p'}} ( \| q^\eps (\cdot, t)\|_{L^p} + \| \omega^\eps(\cdot, t) \|_{L^p} + \| \nabla \boldsymbol{u}^\eps(\cdot, t) \|_{L^p} ), \\
\| \boldsymbol{F}^\eps (\cdot, t)\|_{L^r} &\leq \| \boldsymbol{u}^\eps(\cdot, t) \|_{L^\infty} ( \| h^\eps \|_{L^{s_2}} \| q^\eps (\cdot, t)\|_{L^p} + \| \omega^\eps(\cdot, t) \|_{L^r} + \| \nabla \boldsymbol{u}^\eps (\cdot, t)\|_{L^r} ) , \\
\|\nabla \boldsymbol{F}^\eps(\cdot, t) \|_{L^r} &\leq \|\nabla \boldsymbol{u}^\eps (\cdot, t) \|_{L^{2r}}^2
+ \|\boldsymbol{u}^\eps (\cdot, t)\|_{L^\infty} \left( \| \nabla^2 \boldsymbol{u}^\eps(\cdot, t) \|_{L^r} + \|\nabla h^\eps \|_{L^{s_2}} \| q^\eps(\cdot, t) \|_{L^p} \right).
\end{align*}
Note that \eqref{est:u-eps-infty}, \eqref{est:local-1} and \eqref{est:local-2} imply that $\| \boldsymbol{u}^\eps (\cdot, t)\|_{L^\infty} \leq C_\eps$, $\| \boldsymbol{u}^\eps (\cdot, t)\|_{L^{p'}} \leq C_\eps$ and
\begin{equation*}
\| \nabla \boldsymbol{u}^\eps(\cdot, t) \|_{L^{2r}} \leq C \| \omega^\eps(\cdot, t) \|_{L^{2r}} \leq C \eps^{-2 + 2/{s_3}} \| h \|_{L^{s_3}} \| q_0 \|_{L^p}
\end{equation*}
for $s_3 \in (1,\infty)$ satisfying $1 + 1/{2r} = 1/{s_3} + 1/p$, respectively. Thus, we find
\begin{equation*}
\| \boldsymbol{F}^\eps (\cdot, t) \|_{L^{s_1}} \leq C_\eps, \quad \| \boldsymbol{F}^\eps (\cdot, t) \|_{L^r} \leq C_\eps, \quad \|\nabla \boldsymbol{F}^\eps(\cdot, t) \|_{L^r} \leq C_\eps,
\end{equation*}
which yields $P^\eps \in C([0,T]; W^{2,r}(\mathbb{R}^2))$ for any $r \in (2, p^\ast)$, so that $P^\eps \in C([0,T]; C^1(\mathbb{R}^2))$ holds.


Next, we see the convergence of $P^\eps$. Recall that $\boldsymbol{u}^\eps$ converges to $\boldsymbol{u} \in L^\infty(0,T; L^{p^\ast}(\mathbb{R}^2))$ strongly in $C([0,T]; L^r_{\mathrm{loc}}(\mathbb{R}^2))$ for any $r \in [1, p^\ast)$. Fix a constant $r \in [2, p^\ast)$ and define
\begin{equation*}
P(\boldsymbol{x},t) \equiv - \nabla \cdot \int_{\mathbb{R}^2} \nabla G(\boldsymbol{x} - \boldsymbol{y}) \cdot (\boldsymbol{u} \otimes \boldsymbol{u})(\boldsymbol{y},t) d\boldsymbol{y}.
\end{equation*}
Then, $P \in L^\infty(0,T; L^{p^\ast/2}(\mathbb{R}^2))$ holds since the Calder\'{o}n-Zygmund inequality and \eqref{ineq:HLS} yield
\begin{equation*}
\| P(\cdot,t) \|_{L^{p^\ast/2}} \leq C \| (\boldsymbol{u} \otimes \boldsymbol{u})(\cdot,t) \|_{L^{p^\ast/2}} \leq C \| \boldsymbol{u} (\cdot,t) \|_{L^{p^\ast}}^2 \leq C \| q_0 \|_{L^p}^2.
\end{equation*}
Note that
\begin{align*}
P^\eps(\boldsymbol{x},t)  - P(\boldsymbol{x},t) &= - \nabla \cdot \int_{\mathbb{R}^2} \nabla G(\boldsymbol{x} - \boldsymbol{y}) \cdot \left[ (\boldsymbol{u}^\eps \otimes \boldsymbol{u}^\eps)(\boldsymbol{y},t) - (\boldsymbol{u} \otimes \boldsymbol{u})(\boldsymbol{y},t) \right] d\boldsymbol{y} \\
& \quad + \int_{\mathbb{R}^2} \boldsymbol{K}(\boldsymbol{x} - \boldsymbol{y}) \cdot \left[  (h^\eps \ast (\boldsymbol{u}^\varepsilon q^\varepsilon))(\boldsymbol{y},t)  - (\boldsymbol{u}^\varepsilon \omega^\varepsilon)(\boldsymbol{y},t) \right] d\boldsymbol{y} \\
& \equiv I_1^\eps(\boldsymbol{x},t) + I_2^\eps(\boldsymbol{x},t).
\end{align*}
For any fixed $R > 0$, it follows that
\begin{align*}
\| I_1^\eps(\cdot,t) \|_{L^{r/2}(B_R)} &\leq \left\| \nabla \hspace{-0.2mm} \cdot \hspace{-0.3mm} \int \nabla G(\cdot - \boldsymbol{y}) \cdot \left[ (\boldsymbol{u}^\eps \otimes \boldsymbol{u}^\eps) - (\boldsymbol{u} \otimes \boldsymbol{u}) \right](\boldsymbol{y},t) \chi_{B_{R'}}(\boldsymbol{y}) d\boldsymbol{y} \right\|_{L^{r/2}} \\
& + \left\| \int (\nabla^2 G)(\cdot - \boldsymbol{y}) : \left[ (\boldsymbol{u}^\eps \otimes \boldsymbol{u}^\eps) - (\boldsymbol{u} \otimes \boldsymbol{u}) \right](\boldsymbol{y},t) \chi_{\mathbb{R}^2 \setminus B_{R'}}(\boldsymbol{y}) d\boldsymbol{y} \right\|_{L^{r/2}(B_R)}
\end{align*}
for any $R' > 2 R$. Thus, we have
\begin{align*}
\| I_1^\eps(\cdot,t) \|_{L^{r/2}(B_R)} &\leq C \| ((\boldsymbol{u}^\eps \otimes \boldsymbol{u}^\eps) - (\boldsymbol{u} \otimes \boldsymbol{u}) )(\cdot,t) \|_{L^{r/2}(B_{R'})} \\
&\quad + |B_R|^{2/r} \| (\nabla^2 G) \chi_{\mathbb{R}^2 \setminus B_{R'/2}} \|_{L^{(p^\ast/2)'}}  \| ( (\boldsymbol{u}^\eps \otimes \boldsymbol{u}^\eps) - (\boldsymbol{u} \otimes \boldsymbol{u}) )(\cdot,t) \|_{L^{p^\ast/2}} \\
&\leq C (\| \boldsymbol{u}^\eps (\cdot,t) \|_{L^r} + \| \boldsymbol{u} (\cdot,t) \|_{L^r}) \| (\boldsymbol{u}^\eps - \boldsymbol{u})(\cdot,t) \|_{L^r(B_{R'})} \\
&\quad + C |B_R|^{2/r} (R')^{2(1 - 2/p)} (\| \boldsymbol{u}^\eps (\cdot,t) \|_{L^{p^\ast}}^2 + \| \boldsymbol{u} (\cdot,t) \|_{L^{p^\ast}}^2)  \\
& \leq C(q_0, R) \left( \| (\boldsymbol{u}^\eps - \boldsymbol{u})(\cdot,t) \|_{L^r(B_{R'})} + (R')^{2(1 - 2/p)} \right),
\end{align*}
where $C(q_0,R)$ is the constant depending on $\| q_0\|_{L^1}$, $\| q_0\|_{L^p}$ and $R$. To estimate $I_2^\eps$, we introduce the following lemma, whose proof is similar to Lemma II.1 in \cite{Diperna(a)}.

\begin{lemma} \label{lem1}
Let $\varphi$ satisfy $w_1 \varphi \in L^1(\mathbb{R}^2)$. For any $\boldsymbol{u} \in W^{1,\beta}(\mathbb{R}^2)$ and $q \in L^\gamma(\mathbb{R}^2)$, we have
\begin{equation*}
\| \varphi \ast (\boldsymbol{u} q) - \boldsymbol{u} \cdot (\varphi \ast q) \|_{L^\alpha} \leq \| w_1 \varphi \|_{L^1} \| \nabla \boldsymbol{u} \|_{L^\beta} \| q \|_{L^\gamma},
\end{equation*}
where $1 / \alpha = 1/ \beta + 1/ \gamma$.
\end{lemma}
\begin{proof}
It follows that
\begin{align*}
&\varphi \ast (\boldsymbol{u} q) - \boldsymbol{u} \cdot (\varphi \ast q) \\
&= \int_{\mathbb{R}^2} \varphi (\boldsymbol{x} - \boldsymbol{y}) \cdot \boldsymbol{u}(\boldsymbol{y}) q(\boldsymbol{y}) d \boldsymbol{y} -  \boldsymbol{u}(\boldsymbol{x}) \cdot \int_{\mathbb{R}^2}  \varphi (\boldsymbol{x} - \boldsymbol{y}) q(\boldsymbol{y}) d \boldsymbol{y} \\
& = \int_{\mathbb{R}^2} \varphi (\boldsymbol{x} - \boldsymbol{y}) \cdot  \left( \boldsymbol{u}(\boldsymbol{y})  -  \boldsymbol{u}(\boldsymbol{x}) \right) q(\boldsymbol{y}) d \boldsymbol{y}\\
& = \int_{\mathbb{R}^2} \varphi (\boldsymbol{y}) \cdot  \left( \boldsymbol{u}(\boldsymbol{x} - \boldsymbol{y})  -  \boldsymbol{u}(\boldsymbol{x}) \right) q(\boldsymbol{x} - \boldsymbol{y}) d \boldsymbol{y}.
\end{align*}
Thus, we find
\begin{align*}
& \|  \varphi \ast (\boldsymbol{u} q) - \boldsymbol{u} \cdot (\varphi \ast q) \|_{L^\alpha} \\
& \leq  \int_{\mathbb{R}^2} | \varphi (\boldsymbol{y})| \left( \int_{\mathbb{R}^2} | \boldsymbol{u}(\boldsymbol{x} - \boldsymbol{y}) - \boldsymbol{u}(\boldsymbol{x}) |^\alpha |q(\boldsymbol{x} - \boldsymbol{y})|^\alpha d \boldsymbol{x} \right)^{1/\alpha} d \boldsymbol{y} \\
& \leq  \| q \|_{L^\gamma} \int_{\mathbb{R}^2} |\varphi (\boldsymbol{y})| \left( \int_{\mathbb{R}^2} | \boldsymbol{u}(\boldsymbol{x} - \boldsymbol{y}) - \boldsymbol{u}(\boldsymbol{x}) |^\beta d \boldsymbol{x} \right)^{1/\beta} d \boldsymbol{y}.
\end{align*}
Since we have $\| \boldsymbol{u}^\varepsilon(\cdot - \boldsymbol{y}) -  \boldsymbol{u}^\varepsilon(\cdot) \|_{L^\beta} \leq |\boldsymbol{y}| \| \nabla \boldsymbol{u}^\varepsilon \|_{L^\beta}$, we obtain the desired estimate.
\end{proof}

It follows from Lemma~\ref{lem1} that
\begin{equation*}
\|  (h^\eps \ast (\boldsymbol{u}^\varepsilon q^\varepsilon))(\cdot,t)  - (\boldsymbol{u}^\varepsilon \omega^\varepsilon)(\cdot,t) \|_{L^\alpha} \leq \eps \| w_1 h \|_{L^1} \| \nabla \boldsymbol{u}^\eps(\cdot,t) \|_{L^\beta} \| q_0 \|_{L^p},
\end{equation*}
where $1/\alpha = 1/\beta + 1/p$. Considering \eqref{ineq:CZ}, we find $ \| \nabla \boldsymbol{u}^\eps (\cdot,t)\|_{L^\beta} \leq C \eps^{-2 + 2/s} \| h\|_{L^s} \| q_0 \|_{L^p}$ for $s \in (1, \infty)$ satisfying $1 + 1/\beta = 1/s + 1/p$, that is, $1/s = 1 + 1/\alpha - 2/p$. Thus, we obtain
\begin{equation}
\|  (h^\eps \ast (\boldsymbol{u}^\varepsilon q^\varepsilon))(\cdot,t)  - (\boldsymbol{u}^\varepsilon \omega^\varepsilon)(\cdot,t) \|_{L^\alpha} \leq C \eps^{1 + 2/\alpha - 4/p} \| q_0 \|_{L^p}^2, \label{est:remainder}
\end{equation}
that is, $1/\alpha > 2/p - 1/2$ is required to show that the right-hand side converges to zero in the $\eps \rightarrow 0$ limit. We have
\begin{align*}
\| I_2^\eps(\cdot,t) \|_{L^{r/2}(B_R)} &\leq \| \boldsymbol{K} \chi_{B_1} \|_{L^s} \| (h^\eps \ast (\boldsymbol{u}^\varepsilon q^\varepsilon))(\cdot,t)  - (\boldsymbol{u}^\varepsilon \omega^\varepsilon)(\cdot,t)  \|_{L^{r_1}} \\
& \quad + C_R \| \boldsymbol{K} \chi_{\mathbb{R}^2 \setminus B_1} \|_{L^{r'_2}} \| (h^\eps \ast (\boldsymbol{u}^\varepsilon q^\varepsilon))(\cdot,t)  - (\boldsymbol{u}^\varepsilon \omega^\varepsilon)(\cdot,t)  \|_{L^{r_2}}
\end{align*}
for $r_1$, $s \in [1,\infty)$ satisfying $1 + 2/r = 1/s + 1/{r_1}$ and $r_2 \in [1,2)$. Owing to $2/p - 1/2 < 5/6$, we have $1/{r_2} > 2/p - 1/2$ for any $r_2 \in [1,6/5)$ and thus the second term converges to zero as $\eps \rightarrow 0$. As for the first term, $1 \leq s < 2$ and $1/{r_1} > 2/p - 1/2$ are required to show the convergence. Set $1/s = 1/2 + \delta_1$ for $\delta_1 \in (0,1/2]$. Since there exists $\delta_2 > 0$ such that $1/r = 1/{p^\ast} + \delta_2$, we have
\begin{equation*}
\frac{1}{r_1} = 1 + \frac{2}{r} - \frac{1}{s} = \frac{2}{p} - \frac{1}{2} + 2 \delta_2 - \delta_1.
\end{equation*}
Taking sufficiently small $\delta_1$ satisfying $2 \delta_2 > \delta_1$, we find $1/{r_1} > 2/p - 1/2$. Hence, we conclude $\| I_2^\eps \|_{L^\infty(0,T;L^{r/2}(B_R))} \rightarrow 0$ as $\eps \rightarrow 0$. Summarizing the above estimates, we obtain
\begin{align*}
\| P^\eps - P \|_{L^\infty(0,T; L^{r/2}(B_R))} &\leq C(q_0, R) \left( \| \boldsymbol{u}^\eps - \boldsymbol{u} \|_{L^\infty(0,T; L^r(B_{R'}))} + (R')^{2(1 - 2/p)} + \eps^\alpha \right)
\end{align*}
for some $\alpha > 0$ and sufficiently large $R' >0$. Taking the $\eps \rightarrow 0$ limit and the $R' \rightarrow \infty$ limit in order, we find $P^\eps \rightarrow P$ in $L^\infty(0,T; L^{r/2}_{\mathrm{loc}}(\mathbb{R}^2))$ for any $r \in [2, p^\ast)$.

Finally, we show the local energy balance. Considering the regularities of $\boldsymbol{u}^\eps$ and $P^\eps$, we find from \eqref{eq:u-eps} that $\partial_t \boldsymbol{u}^\eps$ is continuous in $\overline{\Omega_T}$. Multiplying the equation \eqref{eq:u-eps} by $\boldsymbol{u}^\eps$, we have
\begin{equation*}
\partial_t \left( \frac{|\boldsymbol{u}^\varepsilon|^2}{2} \right) + \nabla \cdot \left[ \boldsymbol{u}^\varepsilon \left( \frac{|\boldsymbol{u}^\varepsilon|^2}{2} + P^\eps \right) \right] =  - \boldsymbol{u}^\varepsilon \cdot \left( h^\eps \ast ((\boldsymbol{u}^\varepsilon)^\perp q^\varepsilon) - (\boldsymbol{u}^\varepsilon)^\perp \omega^\varepsilon  \right).
\end{equation*}
Recall $\boldsymbol{u}^\eps \rightarrow \boldsymbol{u}$ in $L^\infty(0,T;L^{r_1}_{\mathrm{loc}}(\mathbb{R}^2))$ and $P^\eps \rightarrow P$ in $L^\infty(0,T;L^{r_2}_{\mathrm{loc}}(\mathbb{R}^2))$ for any $r_1 \in [1, p^\ast)$ and $r_2 \in [1, p^\ast/2)$. In a similar way shown in \cite{Cheskidov(b)}, we obtain
\begin{align*}
&\partial_t \left( \frac{|\boldsymbol{u}^\varepsilon|^2}{2} \right) \ \longrightarrow\ \partial_t \left( \frac{|\boldsymbol{u}|^2}{2} \right), \\
&\nabla \cdot \left[ \boldsymbol{u}^\varepsilon \left( \frac{|\boldsymbol{u}^\varepsilon|^2}{2} + P^\eps \right) \right]  \ \longrightarrow\  \nabla \cdot \left[ \boldsymbol{u} \left( \frac{|\boldsymbol{u}|^2}{2} + P \right) \right],
\end{align*}
in the sense of distributions. Indeed, setting $p_1 \equiv (p^\ast)'$ and $p_2 \equiv (p^\ast/2)'$ for simplicity, we have
\begin{align*}
\left\| \left( \frac{|\boldsymbol{u}^\varepsilon|^2}{2} - \frac{|\boldsymbol{u}|^2}{2} \right)(\cdot,t) \right\|_{L^{p_1}(B_R)} &\leq \frac{1}{2} \| (\boldsymbol{u}^\varepsilon - \boldsymbol{u})(\cdot,t) \|_{L^{p_2}(B_R)} \left( \| \boldsymbol{u}^\varepsilon (\cdot,t) \|_{L^{p^\ast}} + \| \boldsymbol{u} (\cdot,t) \|_{L^{p^\ast}} \right) \\
&\leq C \| \boldsymbol{u}^\varepsilon - \boldsymbol{u} \|_{L^\infty(0,T;L^{p_2}(B_R))} \| q_0 \|_{L^p}
\end{align*}
and
\begin{align*}
&\left\| \left[ \boldsymbol{u}^\varepsilon \left( \frac{|\boldsymbol{u}^\varepsilon|^2}{2} + P^\eps \right)  -  \boldsymbol{u} \left( \frac{|\boldsymbol{u}|^2}{2} + P \right) \right](\cdot,t) \right\|_{L^1(B_R)} \\
&\leq \frac{1}{2} \| (\boldsymbol{u}^\varepsilon - \boldsymbol{u})(\cdot,t) \|_{L^{p_2}(B_R)}  \| \boldsymbol{u}^\varepsilon (\cdot,t)\|_{L^{p^\ast}}^2  + \| \boldsymbol{u}(\cdot,t) \|_{L^{p^\ast}} \left\| \left( \frac{|\boldsymbol{u}^\varepsilon|^2}{2} - \frac{|\boldsymbol{u}|^2}{2} \right)(\cdot,t) \right\|_{L^{p_1}(B_R)} \\
& \quad + \| (\boldsymbol{u}^\varepsilon - \boldsymbol{u})(\cdot,t) \|_{L^{p_2}(B_R)} \| P^\varepsilon (\cdot,t) \|_{L^{p^\ast/2}} + \| \boldsymbol{u}(\cdot,t) \|_{L^{p^\ast}} \| (P^\varepsilon - P) (\cdot,t) \|_{L^{p_1}(B_R)} \\
& \leq C \| \boldsymbol{u}^\varepsilon - \boldsymbol{u} \|_{L^\infty(0,T;L^{p_2}(B_R))} \| q_0 \|_{L^p}^2 + C  \| P^\varepsilon - P \|_{L^\infty(0,T;L^{p_1}(B_R))} \| q_0 \|_{L^p}.
\end{align*}
Owing to $p_1 \in (1, 6/5)$ and $p_2 \in (1,3/2)$, the desired convergences hold. It remains to show that
\begin{equation*}
R^\eps \equiv \boldsymbol{u}^\varepsilon \cdot \left( h^\eps \ast ((\boldsymbol{u}^\varepsilon)^\perp q^\varepsilon) - (\boldsymbol{u}^\varepsilon)^\perp \omega^\varepsilon  \right)
\end{equation*}
converges to zero in $L^\infty(0,T; L^1(\mathbb{R}^2))$. It follows from \eqref{est:remainder} that
\begin{equation}
\| R^\eps(\cdot,t)  \|_{L^1} \leq \| \boldsymbol{u}^\varepsilon(\cdot,t) \|_{L^{p^\ast}} \| (h^\eps \ast (\boldsymbol{u}^\varepsilon q^\varepsilon) - \boldsymbol{u}^\varepsilon \omega^\varepsilon )(\cdot,t) \|_{L^{(p^\ast)'}} \leq C \eps^\alpha \| q_0 \|_{L^p}^3, \label{est:R-eps}
\end{equation}
in which $\alpha$ is given by
\begin{equation*}
\alpha = 1 + \frac{2}{(p^\ast)'} - \frac{4}{p} = 6\left( \frac{2}{3} - \frac{1}{p} \right).
\end{equation*}
Thus, we find $\| R^\eps \|_{L^\infty(0,T;L^1(\mathbb{R}^2))} \rightarrow 0$ as $\eps \rightarrow 0$ for $3/2< p < 2$.

\subsection{Proof of Theorem~\ref{thm-main2}}

The proof for the convergence to the Euler equations is the same as Section~\ref{proof:conv-Euler}, see Remark~\ref{rem:conv-Euler}. We show the convergence of the energy dissipation rate. Similarly to the case of $3/2 < p <2$, we have
\begin{equation*}
| \mathscr{D}_E^\varepsilon(t) | \leq \frac{1}{2} \int_{\mathbb{R}^2} |(\nabla H_G^\varepsilon) (\boldsymbol{y})| \| \boldsymbol{u}^\varepsilon(\cdot, t) - \boldsymbol{u}^\varepsilon(\cdot - \boldsymbol{y}, t) \|_{L^3} \| q^\varepsilon(\cdot,t) q^\varepsilon(\cdot - \boldsymbol{y},t) \|_{L^{3/2}} d \boldsymbol{y}.
\end{equation*}
Thus, we find
\begin{equation*}
| \mathscr{D}_E^\varepsilon(t) | \leq C \| q^\varepsilon(\cdot,t) \|_{L^{3/2}}^2  \| w_\alpha \nabla H_G^\varepsilon \|_{L^3}.
\end{equation*}
Here,
\begin{align*}
\| w_\alpha \nabla H_G^\varepsilon \|_{L^3} &\leq C \left( \int_{\mathbb{R}^2}  \frac{|\boldsymbol{x}|^{3(\alpha-1)}}{(|\boldsymbol{x}|/\eps + 1)^3} d\boldsymbol{x} \right)^{1/3}  = C \eps^{\alpha -1/3} \left(  \int_0^\infty \frac{r^{3\alpha - 2}}{(r + 1)^3} dr \right)^{1/3},
\end{align*}
and the right-hand side is finite for $\alpha \in (1/3,1]$. Thus, we obtain
\begin{equation*}
| \mathscr{D}_E^\varepsilon(t) | \leq C \eps^{\alpha -1/3} \| q_0 \|_{L^{3/2}}^2,
\end{equation*}
so that $\| \mathscr{D}_E^\varepsilon \|_{L^\infty(0,T)}$ converges to zero in the $\eps \rightarrow 0$ limit.

As for the local energy balance, it is easily confirmed that, except for the estimate \eqref{est:R-eps}, the proof is the same as Section~\ref{proof:balance}. Thus, we show that $\| R^\eps \|_{L^\infty(0,T;L^1(\mathbb{R}^2))}$ converges to zero in the $\eps \rightarrow 0$ limit. We have
\begin{equation*}
\| R^\eps(\cdot,t)  \|_{L^1} \leq \| \boldsymbol{u}^\varepsilon(\cdot,t) \|_{L^\infty} \| (h^\eps \ast (\boldsymbol{u}^\varepsilon q^\varepsilon) - \boldsymbol{u}^\varepsilon \omega^\varepsilon )(\cdot,t) \|_{L^1}.
\end{equation*}
It follows from \eqref{est:u-eps-infty} that
\begin{equation*}
\| \boldsymbol{u}^\varepsilon(\cdot,t) \|_{L^\infty} \leq C \left( \| \omega^\varepsilon(\cdot,t) \|_{L^r} + \| \omega^\varepsilon(\cdot,t) \|_{L^1} \right) \leq C \eps^{-2 + 2 /s} (\| q_0 \|_{L^{3/2}} + \| q_0 \|_{L^1})
\end{equation*}
for any $r \in (2, \infty]$ and $s \in (6/5,3)$ satisfying $1 + 1/r = 1/s + 2/3$. On the other hand, the proof of Lemma~\ref{lem1} implies
\begin{align*}
\| (h^\eps \ast (\boldsymbol{u}^\varepsilon q^\varepsilon) - \boldsymbol{u}^\varepsilon \omega^\varepsilon )(\cdot,t) \|_{L^1} &\leq \| q_0 \|_{L^{3/2}} \int_{\mathbb{R}^2} |h^\eps (\boldsymbol{y})| \| \boldsymbol{u}^\eps(\cdot - \boldsymbol{y},t) - \boldsymbol{u}^\eps(\cdot,t) \|_{L^3} d \boldsymbol{y} \\
& \leq C \| q_0 \|_{L^{3/2}} \int_{\mathbb{R}^2} |\boldsymbol{y}|^\alpha |h^\eps (\boldsymbol{y})| d \boldsymbol{y} \\
&= C \eps^{\alpha} \| q_0 \|_{L^{3/2}} \| w_\alpha h\|_{L^1}.
\end{align*}
Thus, we obtain
\begin{equation*}
\| R^\eps(\cdot,t)  \|_{L^1} \leq C(q_0,T) \eps^{2 /r - 4/3 + \alpha}.
\end{equation*}
Since $r \in (2, \infty]$ is an arbitrary constant, we set $2/r = 1 - \delta$ for  $\delta \in (0, \alpha - 1/3)$. Then, we have
\begin{equation*}
\frac{2}{r} - \frac{4}{3} + \alpha = \left(\alpha - \frac{1}{3} \right) - \delta > 0,
\end{equation*}
so that we conclude $\| R^\eps \|_{L^\infty(0,T;L^1(\mathbb{R}^2))} \rightarrow 0$ as $\eps \rightarrow 0$.

\appendix
\section{Properties of the auxiliary function $H_G^\eps$} \label{appendix_A}

We see detailed properties of $H_G^\eps$ defined by \eqref{def-H_G-eps}. Note that
\begin{equation*}
H_G^\eps =  h^\varepsilon \ast \left(h^\eps \ast G - G \right)  =  G \ast \left( h^\eps \ast h^\eps - h^\eps \right) = G \ast H^\eps,
\end{equation*}
where $H^\eps(\boldsymbol{x}) \equiv \eps^{-2} H (\eps^{-1} \boldsymbol{x})$ and $H \equiv h \ast h - h$. Since $h$ is a radial function, we find that $H$, $H^\eps$ and $H_G^\eps$ are also radially symmetric. To emphasize that, those radial functions are denoted by $H_r$, $H_r^\eps$ and $H_{G,r}^\eps$. In what follows, we assume that $h \in L^1(\mathbb{R}^2)$ satisfies $w_3 h$, $w_\alpha h \in L^\infty(\mathbb{R}^2)$ for some $\alpha \in (0,1)$.

Recall that $G$ is a fundamental solution to the 2D Laplacian. Then, we have $\Delta H_G^\eps = H^\eps$, that is,
\begin{equation*}
\nabla H_G^\eps (\boldsymbol{x}) = (\nabla G) \ast H^\eps = \frac{\boldsymbol{x}}{|\boldsymbol{x}|^2} \int_0^{|\boldsymbol{x}|} s H_r^\eps(s) ds,
\end{equation*}
which is rewritten by
\begin{equation*}
\nabla H_G^\eps (\boldsymbol{x}) = \frac{\boldsymbol{x}}{|\boldsymbol{x}|} (H_{G,r}^\eps)' (|\boldsymbol{x}|), \qquad (H_{G,r}^\eps)' (r) = \frac{1}{r} \int_0^{r} s H_r^\eps(s) ds.
\end{equation*}
Integrating $(H_{G,r}^\eps)'$ on $[0,r]$, we obtain
\begin{align*}
H_{G,r}^\eps(r) - H_{G,r}^\eps(0) &= \int_0^r \frac{1}{s} \int_0^{s} t H_r^\eps(t) dt ds \\
& = \left[ \log{s} \int_0^{s} t H_r^\eps(t) dt \right]_0^r - \int_0^r (\log{s}) s H_r^\eps(s) ds.
\end{align*}
Note that $w_\beta h \in L^\infty(\mathbb{R}^2)$ yields $w_\beta H \in L^\infty(\mathbb{R}^2)$ for any $\beta > 0$. Then, we have
\begin{equation*}
H_{G,r}^\eps(0) = \int_{\mathbb{R}^2} G(\boldsymbol{y}) H^\eps(\boldsymbol{y}) d \boldsymbol{y} = \int_0^\infty (\log{s}) s H^\eps_r(s) ds,
\end{equation*}
and the right-hand side is finite. Thus, we obtain
\begin{align*}
H_{G,r}^\eps(r) &= \log{r} \int_0^{r} s H_r^\eps(s) ds + \int_r^\infty (\log{s}) s H_r^\eps(s) ds \\
& = \log{r} \int_0^{r / \eps } s H_r(s) ds + \int_{r / \eps}^\infty (\log{s}) s H_r(s) ds + \log{\eps} \int_{r / \eps}^\infty s H_r(s) ds.
\end{align*}
Since $\int_{\mathbb{R}^2} h (\boldsymbol{x}) d \boldsymbol{x} = 1$ yields $\int_0^\infty t H_r(t) dt = 0$, we have the following expression of $H_{G,r}^\eps$.
\begin{equation*}
H_{G,r}^\eps(r)  = \log{\left( r /\eps \right)} \int_0^{r / \eps } s H_r(s) ds + \int_{r / \eps}^\infty (\log{s}) s H_r(s) ds.
\end{equation*}

We show that, for any fixed $\eps$, the functions $H_{G,r}^\eps$ and $(H_{G,r}^\eps)'$ belong to $C_0[0, \infty)$. Setting $\rho = r /\eps$ for convenience, we have
\begin{equation*}
H_{G,r}^\eps(r) = \int_0^\rho \log{\left( \frac{\rho}{s} \right)} s H_r(s) ds + \int_0^\infty (\log{s}) s H_r(s) ds.
\end{equation*}
It follows from $w_\alpha H \in L^\infty(\mathbb{R}^2)$ that
\begin{align*}
|H_{G,r}^\eps(r)| = \rho \int_0^{\rho} | H_r(s)| ds + |H_{G,r}^{\eps =1}(0)| \leq C_\alpha \rho^{2 - \alpha} + |H_{G,r}^{1}(0)|.
\end{align*}
Thus, $H_{G,r}^\eps$ is bounded for $\rho \leq 1$. For $\rho > 1$, we have
\begin{equation*}
H_{G,r}^\eps(r)  = \int_\rho^\infty \log{\left( \frac{s}{\rho} \right)} s H_r(s) ds = \int_1^\infty (\log{s}) (\rho s) H_r(\rho s) \rho ds,
\end{equation*}
and
\begin{equation*}
| H_{G,r}^\eps(r) | \leq  \frac{1}{\rho} \| w_3 H \|_{L^\infty} \int_1^\infty \frac{\log{s}}{s^2} ds = \frac{\eps}{r} \| w_3 H \|_{L^\infty}.
\end{equation*}
Thus, we conclude $H_{G,r}^\eps \in C_0 [0, \infty)$. As for the derivative $(H_{G,r}^\eps)'$, note that
\begin{equation*}
(H_{G,r}^\eps)' (r) = \frac{1}{r} \int_0^{r} s H_r^\eps(s) ds = \frac{1}{\eps \rho} \int_0^\rho s H_r (s) ds = \frac{1}{\eps} \int_0^1 (\rho s) H_r (\rho s) ds.
\end{equation*}
Then, we have $|(H_{G,r}^\eps)' (r)| \leq \eps^{-1} \| w_1 H \|_{L^\infty}$, that is, $(H_{G,r}^\eps)' \in L^\infty(0,\infty)$. We show the following decay estimate of $(H_{G,r}^\eps)'$, which is suggested in \cite{Liu}.
\begin{equation}
\left| (H_{G,r}^\eps)' (r) \right| \leq \frac{C}{r} \frac{\eps}{r + \eps}. \label{est:dH_gr-eps}
\end{equation}
Set $(H_{G,r}^\eps)' (r) = I(r/\eps) / r$, where $I(r)$ is defined by
\begin{equation*}
I(r) = \int_0^{r} t H_r(t) dt.
\end{equation*}
Then, it is sufficient to show that $(1 + r) |I(r)|$ is uniformly bounded for $r \in [0,\infty)$. Considering $h \in L^1(\mathbb{R}^2)$ and $\int_0^\infty t H_r(t) dt = 0$, we find $I \in L^\infty(0,\infty)$ and
\begin{align*}
r |I(r)| & = r \left| \int_r^\infty t H_r(t) dt  \right| \leq \| w_3 H \|_{L^\infty} \int_1^\infty \frac{1}{t^2} dt = \| w_3 H \|_{L^\infty}.
\end{align*}
Thus, we obtain \eqref{est:dH_gr-eps}, so that $(H_{G,r}^\eps)' \in C_0[0, \infty)$.

\subsection*{Acknowledgements}
This work was supported by JSPS KAKENHI Grant Number JP19J00064 and JP21K13820.


\end{document}